\documentclass{amsart} 

\usepackage[utf8]{inputenc} 

\usepackage{geometry} 
\usepackage{url}
\usepackage{graphicx} 
\usepackage{mathrsfs}

\usepackage{paralist} 
\usepackage{verbatim} 
\usepackage{subfig} 
\usepackage{color}
\usepackage{mathtools}

\usepackage{amsmath, amsthm, amssymb}
\newtheorem{lem}{Lemma}
\newtheorem{defn}{Definition}
\newtheorem{thm}{Theorem}
\newtheorem{rmk}{Remark}
\newtheorem{cor}{Corollary}
\newtheorem{prop}{Proposition}
\DeclareMathOperator{\Gr}{Gr}
\newcommand{\Lm}{\Lambda}
\usepackage{tikz}
\usetikzlibrary{arrows}
\usetikzlibrary{calc}
\usepackage[all]{xy}

\title{The Purity Conjecture in type $C$}
\author{Rachel Karpman} 
\address{Department of Mathematics, The Ohio State University, Columbus, OH 43210}
\email{karpman.2@osu.edu}
\date{\today}

\begin{document}

\begin{abstract}
A collection $\mathcal{C}$ of $k$-element subsets of $\{1,2,\ldots,m\}$ is \emph{weakly separated} if for each $I, J \in \mathcal{C}$, when the integers $1,2,\ldots,m$ are arranged around a circle, there is a chord separating $I\backslash J$ from $J \backslash I$. Oh, Postnikov and Speyer constructed a correspondence between weakly separated collections which are maximal by inclusion and \emph{reduced plabic graphs}, a class of networks defined by Postnikov which give coordinate charts on the \emph{Grassmannian} of $k$-planes in $m$-space.    As a corollary, they proved Scott's \emph{Purity Conjecture}, which states that a weakly separated collection is maximal by inclusion if and only if it is maximal by size. In this note, we describe maximal weakly separated collections corresponding to \emph{symmetric} plabic graphs, which give coordinate charts on the \emph{Lagrangian Grassmannian}, and prove a symmetric version of the Purity Conjecture.
\end{abstract}

\maketitle
\section{Introduction}

Two $k$-element subsets $I$ and $J$ of $[m] \coloneqq \{1,2,\ldots,m\}$ are \emph{weakly separated} if, when the integers $1,2,\ldots,m$ are arranged around a circle, there is a chord separating $I \backslash J$ from $J \backslash I$. Weak separation was first introduced by Leclerc and Zelevinsky in their study of \emph{quantum flag minors} \cite{LZ98}. In related work, Scott showed that a collection of $k$-element subsets of $[m]$ that are pairwise weakly separated has order at most $k(m-k) + 1$, and made the following \emph{Purity Conjecture}: a weakly separated collection of $k$-elements subsets of $[m]$ is maximal by  \emph{inclusion} if and only if it is maximal by \emph{size} \cite{JS05}.

Weak separation plays an important role in the combinatorics of the \emph{Grassmannian} $\Gr(k,m)$ of $k$-planes in $m$-space. Concretely, $\Gr(k,m)$ is the space of $k \times m$ matrices, modulo row operations.  The \emph{totally positive Grassmannian} $\Gr_{>0}(k,m)$ is the subset of $\Gr(k,m)$ where all $k \times k$ matrix minors are positive.  In his celebrated paper \cite{Pos06}, Postnikov introduced a class of planar, bicolored networks called \emph{plabic graphs}, and defined coordinate charts on $\Gr_{>0}(k,m)$ by assigning weights to the graphs' edges. 

Scott gave a construction which assigns a $k$-element subset of $[m]$ to each face of a plabic graph for $\Gr_{>0}(k,m)$ \cite{JS06}.  The matrix minors with columns indexed by these \emph{face labels} give rational coordinates on $\Gr(k,m)$, which form a \emph{cluster} in a \emph{cluster algebra} structure on the affine coordinate ring of the Grassmannian.  Further, the face labels of a plabic graph form a weakly separated collection which is maximal by size \cite{JS06}.  Oh, Postnikov, and Speyer proved Scott's Purity Conjecture in \cite{OPS15}, by showing that a weakly separated collection is maximal by inclusion if and only if it is the set of face labels of a plabic graph for $\Gr_{>0}(k,m)$.

In \cite{Kar18}, the author extended many results of \cite{Pos06} to the \emph{Lagrangian Grassmannian} $\Lm(2n)$, the space of maximal isotropic subspaces with respect to a symplectic form, using \emph{symmetric plabic graphs}.  A plabic graph is \emph{symmetric} if reflection about a distinguished diameter leaves the uncolored graph unchanged, but \emph{reverses} the colors of vertices. The author defined coordinate charts on the \emph{totally positive part} $\Lm_{>0}(2n)$ of $\Lm(2n)$ using Postnikov's approach, but did not consider face labels of symmetric plabic graphs.  The present paper aims to fill this gap.

After reviewing some combinatorial background, we show in Section \ref{sym_face} that the face labels of symmetric plabic graphs give rational coordinates on $\Lm(2n)$. For an $n$-element subset $I$ of $[2n]$, we define a set $\bar{I}$ which is the ``mirror image'' of $I$.  We call a weakly separated collection \emph{symmetric} if $\bar{I} \in \mathcal{C}$ for each $I \in \mathcal{C}$. In Section \ref{sym_tilings}, we prove that a plabic graph is symmetric if and only if its face labels form a symmetric weakly separated collection.
Our main result states that if $\mathcal{C}$ is a symmetric weakly separated collection of $n$-element subsets of $[2n]$, the following are equivalent:
\begin{enumerate}
\item $\mathcal{C}$ is maximal by inclusion among symmetric weakly separated collections.
\item $\mathcal{C}$ is the set of face labels of a symmetric plabic graph for $\Lm_{>0}(2n)$.
\end{enumerate}
It follows that a symmetric weakly separated collection is maximal by inclusion if and only if it is maximal by size.

In fact, we prove something stronger.  The \emph{totally nonnegative Grassmannian} $\Gr_{\geq 0}(k,m)$ is the subset of $\Gr(k,m)$ where all $k \times k$ minors are nonnegative (as opposed to strictly positive). Postnikov gave a stratification of $\Gr_{\geq 0}(k,m)$ by \emph{positroid cells}, defined in terms of  vanishing minors; the largest positroid cell is simply $\Gr_{>0}(k,m)$.  Plabic graphs may be defined for all positroid cells, and face labellings of plabic graphs for lower-dimensional cells are weakly separated collections which satisfy an additional condition, and which are maximal by inclusion among all such collections \cite{OPS15}.  The \emph{totally nonnegative part} of $\Lm(2n)$ has an analogous stratification, with $\Lm_{>0}(2n)$ being the unique top-dimensional cell, and our results apply to these lower-dimensional cells, as well as to $\Lm_{>0}(2n)$.

\section{Background}
\label{back}
\subsection{Intervals, subsets, and cyclic order on $\{1,2,\ldots,m\}$}

For natural numbers $k \leq m$, let $[m]$ denote the set $\{1,2,\ldots,m\}$, and let ${{[m]}\choose{k}}$ be the set of all $k$-element subsets of $[m]$.  For $a \in [m]$, let $\leq_a$ denote the cyclic shift of the usual linear order on $m$ given by 
 
\begin{equation}
	a < a + 1 < \ldots < m < 1 < \ldots < a-1.
\end{equation}

We extend this to a partial order on ${{[m]}\choose{k}}$ by setting $I \leq_a J$ for
\[I = \{i_1 <_a i_2 <_a \cdots <_a i_k\}\]
\[J = \{j_1 <_a j_2 <_a \cdots <_a j_k\}\]
if $i_{\ell} \leq_a j_{\ell}$ for all $\ell \in [k]$.
For $\ell = 1$, this is the usual Gale order on ${{[m]}\choose{k}}$.
  
For $a,b \in [m]$, we define the \emph{cyclic interval} $[a,b]^{cyc}$ by

\begin{equation}
	[a,b]^{cyc} = \begin{cases}
	\{a,a+1,\ldots,b\} & a \leq b\\
	\{a,a+1,\ldots,n-1,n,1 \ldots,b\} & a > b
	\end{cases}.
\end{equation} 

Note that if the elements of $[m]$ are arranged clockwise around a circle, then elements of $[a,b]$ are consecutive. We say a sequence $a_1,a_2,\ldots,a_{\ell}$ of elements of $[m]$ is \emph{cyclically ordered} if, when the integers $1$ to $m$ are arranged clockwise around a circle, the terms $a_i$ appear in clockwise order.

Fix $n \in \mathbb{N}$.  For $i \in [2n]$, let $i' = 2n - i + 1$.  For each $I \in {{[2n]}\choose{n}}$, we define $\bar{I}$ to be the set 
\begin{equation}
	\bar{I} = [2n] \setminus \{i' \mid i \in I\}.
\end{equation} 

\subsection{Positroids in types $A$ and $C$}

We realize the \emph{Grassmannian} $\Gr(k,m)$ as the space of $k \times m$ matrices, modulo row operations; a matrix corresponds to  the span of its rows.  The $k \times k$ minors of a matrix representative give homogeneous coordinates on $\Gr(k,m)$, called \emph{Pl\"{u}cker coordinates}.  For $I \in {{[n]}\choose{k}}$, let $\Delta_I$ represent the determinant of the matrix minor with columns indexed by $I$.

The \emph{totally nonnegative Grassmannian} $\Gr_{\geq 0}(k,m)$ is the subset of $\Gr(k,m)$ where all Pl\"{u}cker coordinates are nonnegative real numbers, up to multiplication by a common scalar.  For a point $P \in \Gr(k,m)$, the set 
\[\mathcal{M} = \left\{\left.I \in {{[m]}\choose{k}} \right| \Delta_I(P) \neq 0\right\}\]
is called the \emph{matroid} of $P$. If $P$ is a point in $\Gr_{\geq 0}(k,m)$, the matroid of $P$ is a \emph{positroid}.  The subset of $\Gr_{\geq 0}(k,m)$ consisting of all points with positroid $\mathcal{M}$ is a \emph{positroid cell}. Remarkably, positroid cells stratify $\Gr_{\geq 0}(k,m)$, and each positroid cell is homeomorphic to a ball \cite{Pos06}.  

The positroid stratification of $\Gr_{\geq 0}(k,m)$ extends naturally to a stratification of the entire Grassmannian $\Gr(k,m)$ by \emph{positroid varieties}, which has nice geometric and combinatorial properties. For $P \in \Gr(k,m)$, there is a unique smallest positroid $\mathscr{P}$ which contains the matroid $\mathcal{M}$ of $P$; we call $\mathscr{P}$ the \emph{positroid} of $P$.  We have $\mathscr{P}$ equal to $\mathcal{M}$ if $P \in \Gr_{\geq 0}(k,m)$; otherwise $\mathscr{P}$ is strictly larger than $\mathcal{M}$.  The \emph{positroid variety} $\Pi$ corresponding to a positroid $\mathscr{P}$ is the set of all points in $\Gr(k,m)$ with positroid $\mathscr{P}$, and the totally nonnegative part of $\Pi$ is the positroid cell corresponding to $\mathscr{P}$. For details, see \cite{KLS13}.

In \cite{Kar18}, the author extends Postnikov's theory of total nonnegativity to the \emph{Lagrangian Grassmannian} $\Lm(2n)$, the parameter space of maximal isotropic subspaces with respect to a symplectic bilinear form. Concretely, we realize $\Lm(2n)$ as the subvariety of $\Gr(n,2n)$ which is cut out by the linear relations
\begin{equation}
\label{linear} \left\{\Delta_I = \Delta_{\bar{I}} \left| I \in {{[2n]}\choose{n}}\right.\right\}.\end{equation} 
The \emph{symmetric part} of any $\Pi \subseteq Gr(n,2n)$ is the set-theoretic intersection of $\Pi$ with $\Lm(2n)$; that is, the subset of $\Pi$ where the relations \eqref{linear} hold. 
 
The \emph{totally nonnegative} part of $\Lm(2n)$, denoted $\Lm_{\geq 0}(2n)$, is the subset of $\Lm(2n)$ where all Pl\"{u}cker coordinates are positive. A positroid cell (respectively, variety) is \emph{type C} if its symmetric part is nonempty. The symmetric part of a type C positroid cell is a \emph{totally nonnegative cell} in $\Lm_{\geq 0}(2n)$.  Totally nonnegative cells form a stratification of $\Lm_{\geq 0}(2n)$, analogous to the positroid stratification, which extends naturally to a stratification of $\Lm(2n)$ by the symmetric parts of type C positroid varieties. In particular, the symmetric part of $\Pi^C$ is nonempty if and only if the intersection of $\Pi^C$ with $\Gr_{\geq 0}(n,2n)$ is a totally nonnegative cell in $\Lm_{\geq 0}(2n)$.  For details, see \cite{Kar18}. 

\subsection{Grassmann necklaces and decorated permutations}

In this section, we discuss two families of combinatorial objects which are in bijection with positroids.  All of these definitions and results, as well as many additional details, can be found in \cite{Pos06}.

\begin{defn}
A \emph{Grassmann necklace} of type $(k,m)$ is a sequence
$\mathcal{I} = (I_1,\ldots,I_m)$ of elements of ${{[m]}\choose{k}}$ such that $I_{i+1}$ contains $I_i \backslash \{i\}$ for all $i \in [m]$, where indices are taken modulo $m$.
\end{defn}

Note that if $i \not\in I_{i}$, we must have $I_{i+1} = I_i$.  
For a Grassmann necklace $\mathcal{I}$, define
\[\mathcal{M}_{\mathcal{I}} \coloneqq \left\{\left.J \in {{[m]}\choose{n}} \right| J \geq_i I_i \text{ for all }i \in [m]\right\}.\]
Remarkably, $\mathcal{M}_{\mathcal{I}}$ is a positroid, with $\mathcal{I} \subseteq \mathcal{M}_{\mathcal{I}}$, and every positroid arises uniquely in this way. 

\begin{defn}
A \emph{decorated permutation} is a permutation $f \in S_m$, with fixed points colored black or white.  
\end{defn}

To obtain a Grassmann necklace from a decorated permutation $f$, for each $a \in [m]$, let
\[I_a = \{i \in [m]  \mid i \leq_a f^{-1}(i) \text{ or $i$ is a white fixed point of $f$} \}.\]
Then $\mathcal{I} = (I_1,\ldots,I_m)$ is a Grassmann necklace, and this construction gives a bijection from Grassmann necklaces to decorated permutations.  We say $f$ has type $(k,m)$ if the associated Grassmann necklace has type $(k,m)$.

\begin{defn}
[\cite{Pos06}, Section 17] For $i,j \in [m]$, we say that $(i,j)$ forms an \emph{alignment} of the decorated permutation $f$ if the numbers $i$, $f(i)$, $f(j)$ and $j$ appear in cyclic order, and are all distinct.
\end{defn}

Note that the order of $i$ and $j$ matters for this definition; if $(i,j)$ is an alignment, it is not the case that $(j,i)$ is also an alignment.

\subsection{Plabic graphs}

A \emph{plabic graph} is a planar graph embedded in a disk, with each vertex colored black or white.  The boundary vertices are numbered $1,2,\ldots,n$ in clockwise order, and all boundary vertices have degree one. We call the edges adjacent to boundary vertices \emph{legs} of the graph. 

Postnikov defines a collection of paths and cycles in $G$, called \emph{trips}, as follows \cite{Pos06}. Start by traversing a leg of the graph, starting on the boundary, then proceed according to the \emph{rules of the road}: turn (maximally) left at every white internal vertex, and (maximally) right at every black internal vertex. Continuing in this fashion, we eventually reach a boundary vertex, at which point the trip ends.  See Figure \ref{trip} for an example.  We repeat this process for every boundary vertex of the boundary.

\begin{figure}
\begin{center}
\includegraphics[trim={2in 7.25in 2in 1in},clip]{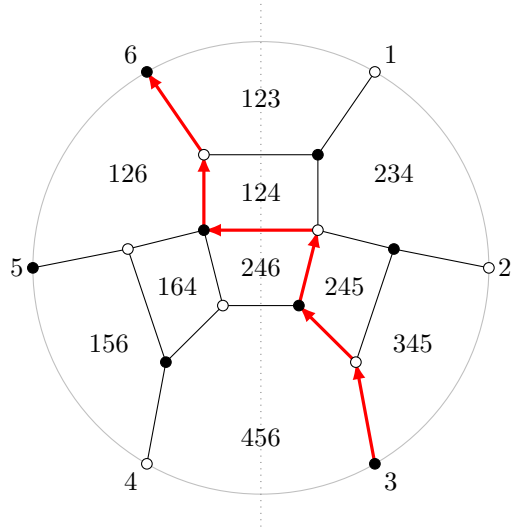}
\end{center}
\caption{A plabic graph with face labels.  Arrows show the trip from $3$ to $6$.}
\label{trip}
\end{figure}

In the remainder of this paper, we restrict our attention to \emph{reduced} plabic graphs.  Postnikov defined reduced plabic graphs in terms of certain local transformations, and gave a criterion for a plabic graph $G$ to be reduced \cite[Section 13]{Pos06}. We take this criterion as the \emph{definition} of a reduced graph.

\begin{defn}
A plabic graph $G$ is \emph{reduced} if it satisfies the following criteria:
\begin{enumerate}
\item The union of the trips defined above covers each edge of the graph exactly twice, once in each direction.
\item $G$ has no leaves, except perhaps some which are adjacent to boundary vertices. 
\item \label{norep} No trip uses the same edge twice, once in each direction, unless that trip starts (and ends) at a boundary vertex connected to a leaf.
\item No trips $T_1$ and $T_2$ share two edges $e_1$ and $e_2$ such that $e_1$ comes before $e_2$ in both trips.  
\end{enumerate}
\end{defn}

Given a plabic graph $G$ with $m$ boundary vertices, we define the \emph{trip permutation} 
\begin{math} f\end{math}
of $G$ by setting 
\begin{math} f(a) = b\end{math}
if the trip that starts at boundary vertex $a$ ends at boundary vertex $b$. Notice that if $f$ has a fixed point at $a$, then $G$ must have a boundary leaf at $a$, to avoid violating Condition \ref{norep} above.  Hence we can define the \emph{decorated trip permutation} of $G$ by coloring each fixed point of $f$ either black or white, depending on the color of the corresponding boundary leaf.

\subsection{Face labels for plabic graphs.}
\label{faces}

Let $G$ be a reduced plabic graph with $m$ boundary vertices, and suppose the decorated permutation $f$ of $G$ has type $(k,m)$.  
Let $\mathcal{M} \subseteq {{[m]}\choose{k}}$ be the positroid associated to $f$.  
For each $i \in [m]$, label each face $F$ of $G$ with an $i$ if $F$ is to the \emph{left} of the trip which \emph{ends} at boundary vertex $i$.  See Figure \ref{trip} for an example.  Once all trips are accounted for, each face is labeled with an element of ${{[m]}\choose{k}}$, which is contained in $\mathcal{M}$ \cite{JS06}. Moreover, the faces adjacent to the boundary of $G$ are labeled with the elements of the Grassmann necklace $\mathcal{I}=(I_1,\ldots,I_n)$ of $\mathcal{M}$ in clockwise order; the boundary face between leg $i-1$ and leg $i$ has label $I_i$ where indices are taken modulo $m$ \cite{OPS15}. The Pl\"{u}cker coordinates corresponding to minors with columns indexed by face labels of $G$ give homogeneous coordinates on an open, dense subset of the positroid variety $\Pi$ of $\mathcal{M}$ \cite{MS17}.

More precisely, let $\mathcal{F}$ denote the set of face labels of $G$, and let $\Pi_{\mathcal{F}}$ be the subset of $\Pi$ where the Pl\"{u}cker coordinates indexed by elements of $\mathcal{F}$ are nonzero. A \emph{face weighting} of $G$ is an assignment of a number, or \emph{weight}, to each face of $G$.  Let $(\mathbb{C}^{\times})^{\mathcal{F}}/\mathbb{C}^{\times}$ denote the space of nonzero face weights of $G$, modulo multiplication by a common scalar. Then by \cite[Theorem 7.1]{MS17}, there is an isomorphism
\[\mathbb{F}: \Pi_{\mathcal{F}} \rightarrow (\mathbb{C}^{\times})^{|\mathcal{F}|}/\mathbb{C}^{\times}\]
which takes a point $P \in \Pi_{\mathcal{F}}$ and labels each face of $G$ with the corresponding Pl\"{u}cker coordinate of $P$.  

We note further that $\mathbb{F}$ maps the totally nonnegative part of $\Pi$, which is a positroid cell, isomorphically onto the space of positive face weightings of $G$.  This is implicit, for example, in the statement of Theorem 7.1 from \cite{MS17}, as the maps $\overleftarrow{\mathbb{M}}$, $\overrightarrow{\mathbb{M}}$, $\overleftarrow{\partial}$ and $\overrightarrow{\partial}$ referenced in the theorem statement take positive edge weightings to positive face weightings and vice versa.  For details, see \cite[Section 5.3-5.4]{MS17}.

\subsection{Symmetric plabic graphs}

For $n \in \mathbb{N}$, a plabic graph with $2n$ boundary vertices is a \emph{symmetric plabic graph} 
if it is symmetric as an uncolored network with respect to reflection through a distinguished diameter which intersects that boundary of the disk between vertices $2n$ and $1$, and between vertices $n$ and $n+1$; and each vertex and its mirror image have opposite colors.  By convention, the line of reflection runs vertically down the center of the graph, with vertices $1$ and $2n$ appearing above vertices $n$ and $n+1$, respectively. The plabic graph in Figure \ref{trip} is symmetric.  

\begin{thm}[\cite{KS18}]
Let $\mathcal{M}$ be a positroid of type $(n, 2n)$.  Then the following are equivalent:
\begin{enumerate}
\item The positroid cell corresponding to $\mathcal{M}$ has nonempty intersection with $\Lm(2n)$.
\item For all $i \in [2n]$, the decorated permutation of $f$ of $\mathcal{M}$ satisfies 
\[f(i') = (f(i))'.\]
\item For all $i \in [2n]$, the Grassmann necklace $\mathcal{I}$ of $\mathcal{M}$ satisfies 
\[I_{i'} = \overline{I_{(i + 1)}}.\]
\item There is a symmetric plabic graph with associated positroid $\mathcal{M}.$
\end{enumerate}
\end{thm}

\begin{defn} We say a decorated permutation, Grassmann necklace, or positroid is type C if the associated positroid cell has nonempty intersection with $\Lm(2n)$.
\end{defn}
 
\subsection{Weakly separated collections and plabic tilings}

The following definitions, with additional details, may be found in \cite{OPS15}.

\begin{defn}
Let $I, J \in {{[m]}\choose{k}}$.
We say $I$ and $J$ are \emph{weakly separated} if there do not exist $a,b,a',b'$, cyclically ordered such that $a,a' \in I \backslash J,$ and $b, b' \in J \backslash I$.  Equivalently, $I$ and $J$ are weakly separated if, when the elements of $[m]$ are arranged in a circle, there is a chord separating the elements of $J \backslash I$ from the elements of $I \backslash J$.
\end{defn}

\begin{defn}Let $\mathcal{C} \subseteq {{[m]}\choose{k}}$.  Then $\mathcal{C}$ is a \emph{weakly separated collection} if for each $I, J \in \mathcal{C}$, the sets $I$ and $J$ are weakly separated.  A weakly separated collection is \emph{maximal by inclusion} if it is not contained in any larger weakly separated collection in ${{[m]}\choose{k}}$; equivalently, if for every $I \in {{[m]}\choose{k}}$ with $I \not\in \mathcal{C}$, there is some $J \in \mathcal{C}$ such that $I$ and $J$ are \emph{not} weakly separated.
\end{defn}

\begin{defn}
Let $\mathcal{M}$ be a positroid with Grassmann necklace $\mathcal{I}$, and let $\mathcal{C}$ be a weakly separated collection in ${{[m]}\choose{k}}$. Then $\mathcal{C}$ is a \emph{weakly separated collection in $\mathcal{M}$} if $\mathcal{I} \subseteq \mathcal{C} \subseteq \mathcal{M}$.  A maximal weakly separated collection in $\mathcal{M}$ is \emph{maximal by inclusion} if it is not contained in any larger weakly separated collection in $\mathcal{M}$.
\end{defn}

Our goal is to prove a symmetric version of following theorem.

\begin{thm}[\cite{OPS15}]
\label{ops_main}
Let $\mathcal{M}$ be a positroid, and let $C$ be a weakly separated collection in $\mathcal{M}$.  Then the following are equivalent.
\begin{enumerate}[(1)]
\item $\mathcal{C}$ is maximal by inclusion.
\item $\mathcal{C}$ is maximal by size among all weakly separated collections in $\mathcal{M}$.
\item $\mathcal{C}$ is the set of face labels of a reduced plabic graph with associated positroid $\mathcal{M}$.
\end{enumerate}
\end{thm}

\subsection{Plabic tilings}

Let $\mathcal{M}$ be a positroid, and let $\mathcal{C}$ be a maximal weakly separated collection in $\mathcal{M}$.  Oh, Postnikov and Speyer construct an abstract CW-complex called a \emph{plabic tiling} with vertices indexed by $\mathcal{C}$, then give an embedding of the complex into $\mathbb{R}^2$.  We proceed directly to a description of the embedded tiling, which is all that we need.  See \cite[section 9]{OPS15} for details.

Fix $m$ points $v_1,\ldots,v_m$ which form the vertices of a convex $m$-gon in $\mathbb{R}^2$, numbered in clockwise order.  We construct a two-dimensional CW-complex $\Sigma(\mathcal{C})$ whose vertices are the points 
\[\left\{\left.v_I \coloneqq \sum_{i \in I} v_i \right| I \in \mathcal{C}\right\}.\]

For each $(k-1)$-element subset $K$ of ${{[m]}\choose{k}}$, the \emph{white clique} corresponding to $K$ is the set 

\[\{I \in \mathcal{C} \mid I = K \cup \{a\} \text{ for some } a \in [m]\}.\]

For each $(k+1)$-element subset $L$ of ${{[m]}\choose{k}}$, the \emph{black clique} corresponding to $L$ is \[\{I \in \mathcal{C} \mid I = L \backslash \{b\} \text{ for some } b \in L\}.\]

A clique is \emph{non-trivial} if it contains three or more vertices.  The points $v_I$ corresponding to elements of a non-trivial clique are the vertices of a convex polygon in $\mathbb{R}^2$. We take these polygons to be the two-dimensional cells of $\Sigma(\mathcal{C})$, and the edges of these polygons to be one-dimensional cells.
We also add an edge from $v_I$ to $v_J$ for any $I, J \in \mathcal{C}$ such that all of the following hold:
\begin{enumerate}
\item $I = (J \backslash \{a\}) \cup \{b\}$ for some $a, b \in [m]$.
\item $\{I,J\}$ is the (trivial) white clique of $I \cap J.$
\item $\{I, J\}$ is also the (trivial) black clique of $I \cup J$.
\end{enumerate}

This construction yields a connected CW-complex embedded in the plane. Remarkably, taking the planar dual of $\Sigma(\mathcal{C})$, and coloring each vertex according to the color of the corresponding face, we obtain a reduced plabic graph.  For each $I \in \mathcal{C}$, Scott's construction assigns the label $I$ to the face of of the graph corresponding to $v_I$.
 
\begin{figure}
\begin{center}
\includegraphics[trim={2in 8.25in 2in 1in},clip]{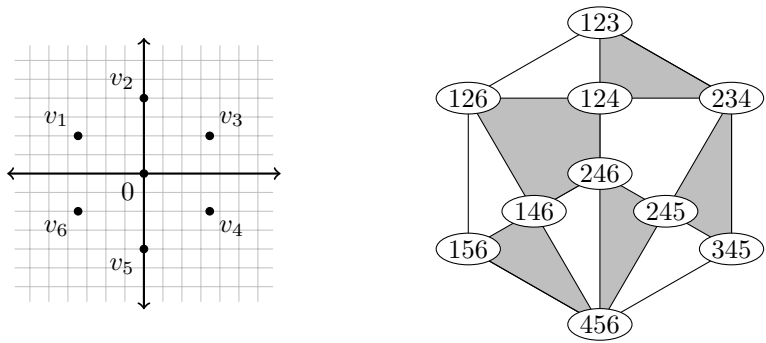}
\end{center}
\caption{A symmetric plabic tiling, defined using the points $v_1,\ldots,v_6$ at left. This tiling is dual to the graph in Figure \ref{trip}.}
\label{numbers}
\end{figure}

Let $G$ be a plabic graph with a boundary leaf at $i \in [m]$.  If the leaf $i$ is white, then $i$ is contained in every face label of $G$. If the leaf is black, then $i$ is contained in none of the face labels.  Hence removing the boundary leaf at $i$ and re-indexing the boundary vertices has no effect on the combinatorial structure of the corresponding tiling, and we may safely reduce to the case of plabic graphs whose trip permutations are fixed-point free.

\section{Symmetric face labels and $\Lm(2n)$}
\label{sym_face}

In this section, we prove that \emph{symmetric face weights} of a symmetric plabic graph give rational coordinates on the symmetric part of the associated positroid variety; and that positive, symmetric face weights give coordinates on the corresponding totally nonnegative cell in $\Lm(2n)$.

\begin{defn}We say a symmetric plabic graph has \emph{symmetric face weights} if each face has the same weight as its mirror image. A graph has \emph{positive} face weights if all of its face weights are positive, up to multiplication by a common nonzero scalar.
\end{defn}

\begin{lem}\label{mirror}
Let $F$ be a face of a symmetric plabic graph $G$, and let $\overline{F}$ be its mirror image.  Then if $F$ has label $I$, $\overline{F}$ has label $\overline{I}$.
\end{lem}

\begin{proof}
Reflecting $G$ through the line of symmetry sends the trip ending at $i$ to the trip ending at $i'$ for all $i \in [2n]$.  Hence $F$ is to the the left of the trip ending at $i$ if and only if $\overline{F}$ is to the \emph{right} of the trip ending at $i'$.  The claim follows.
\end{proof}

\begin{prop}
\label{sym_face_weights}
Let $\Pi$ be a positroid variety of type $C$ in $\Gr(n,2n)$, and let $G$ be a symmetric plabic graph for $\Pi$.  Let $\Pi^C$ be the symmetric part of $\Pi$. Then restricting the face Pl\"{u}cker map $\mathbb{F}$ gives an isomorphism from an open, dense subset of $\Pi^C$ to the space of symmetric face weightings of $G$.  Restricting further to the totally nonnegative part of $\Pi^C$, we obtain an isomorphism to the space of \emph{positive} symmetric face weightings of $G$.
\end{prop}

\begin{proof}
Let $\mathcal{F}$ and $\Pi_{\mathcal{F}}$ be defined as in Section \ref{faces}, and let $\Pi_{\mathcal{F}}^C$ be the symmetric part of $\Pi_{\mathcal{F}}$.  Then $\Pi^C_{\mathcal{F}}$ is an open subset of $\Pi^C$ which is nonempty, since in particular $\Pi^C_{\mathcal{F}}$ contains the totally nonnegative part of $\Pi^C$.  Since $\Pi^C$ is an irreducible variety \cite{Kar18}, it follows that $\Pi^C_{\mathcal{F}}$ is an open dense subset of $\Pi^C$.

Let $W$ denote the space of symmetric weights of $G$, modulo scaling. Clearly, $\mathbb{F}$ maps $\Pi_{\mathcal{F}}^C$ into $W$. Our task is to show that the map is surjective onto $W$.  For this, note that $\mathbb{F}$ is a rational map which is regular on its domain of definition \cite{MS17}.  It follows that $\mathbb{F}$ maps $\Pi_{\mathcal{F}}^C$ isomorphically to a closed subvariety of $W$.  Since $W$ is irreducible, it suffices to show that $\Pi_{\mathcal{F}}^C$ and $W$ have the same dimension.

Let $\ell=\dim(W)$.  Identifying opposite pairs of faces in $G$, and subtracting one to account for scaling, we see that $\ell$ is one less than the number of faces of $G$ which either intersect the distinguished line of reflection, or lie strictly to its left.   But this is the dimension of $\Pi_{\mathcal{F}}^C$, by results of \cite[Section 5]{Kar18}. Hence the dimension of $\Pi_{\mathcal{F}}^C$ is precisely the dimension of $W$, proving the claim.

Next, we show that $\mathbb{F}$ induces an isomorphism from the totally nonnegative part of $\Pi_{\mathcal{F}}^C$ to the space of \emph{positive} symmetric weightings of $G$, modulo scaling. Recall that $\mathbb{F}$ restricts to a bijection from the totally nonnegative part of $\Pi$, which is contained in $\Pi_{\mathcal{F}}$, to the space of positive face weights, modulo scaling. Moreover, we have seen that $\mathbb{F}$ maps $\Pi_{\mathcal{F}}^C$ bijectively to $W$, the space of symmetric weightings, again modulo scaling. Hence restricting $\mathbb{F}$ to the totally nonnegative part of $\Pi_{\mathcal{F}}^C$ gives an isomorphism to the space of positive, symmetric weightings of $G$ modulo scaling, as desired.  
\end{proof}

We have shown that the face Pl\"{u}cker map $\mathbb{F}$ restricts to an isomorphism from a totally nonnegative cell in $\Lm(2n)$ to the space of positive, symmetric face weightings of a symmetric plabic graph. 
Note that since $\mathbb{F}$ is injective on $\Pi_{\mathcal{F}}$, the map $\mathbb{F}$ takes a point $P \in\Pi_{\mathcal{F}}^C$ to a \emph{positive} symmetric weighting of $G$ \emph{only} if $P$ is in the totally nonnegative part of $\Pi^C$.

A collection $\mathcal{C}$ of Pl\"{u}cker coordinates is a \emph{total positivity test} for $\Gr(k,m)$ if all elements of $\mathcal{C}$ are simultaneously positive \emph{only} for points in $\Gr_{> 0}(k,m)$.
Hence any maximal weakly separated collection gives a total positivity test for $\Gr(k,m)$, which is minimal by inclusion.  Similarly, we may say a collection $\mathcal{C}$ of Pl\"{u}cker coordinates is a total positivity test for $\Lm(2n)$ if, whenever the element of $\mathcal{C}$ are simultaneously positive for $P \in \Lm(2n)$, we have $P \in \Lm_{>0}(2n)$. Proposition \ref{faces} shows that we obtain a total positivity test for $\Lm(2n)$ from any symmetric plabic graph for $\Lm_{>0}(2n)$, by for example taking the labels of all faces which either intersect the distinguished diameter, or lie strictly to the left of it. 

\section{Symmetric weakly separated collections}
\label{sym_sep}

\begin{defn}
A weakly separated collection $\mathcal{C}$  is \emph{symmetric} if for all $I \in \mathcal{C}$, we have $\bar{I} \in \mathcal{C}$.
\end{defn}

Hence the face labels of a symmetric plabic graph form a symmetric weakly separated collection.  In particular, if $\mathcal{I}=(I_1,\ldots,I_{2n})$ is a Grassmann necklace of type $C$, then the elements of $\mathcal{I}$ form a symmetric weakly separated collection, since we have $\overline{I_i} = I_{i' + 1}$ for each $i \in [2n]$, with indices taken modulo $2n$. Note that if $I$ is contained in a symmetric weakly separated collection, then $I$ and $\bar{I}$ must be weakly separated.

\begin{defn}
An element $I$ of ${{[2n]}\choose{n}}$ is \emph{admissible} if $I$ and $\bar{I}$ are weakly separated.
\end{defn} 

Let $I \in {{[2n]}\choose{n}}$.  We say that $I$ has a \emph{full pair} at $\{a,a'\}$  if $a,a' \in I$; $I$ has a \emph{half pair} at
 $\{a,a'\}$ if exactly one of
 $\{a,a'\}$ is in $I$; and $I$ has an \emph{empty pair} at
 $\{a,a'\}$ if $a,a' \not\in I$. 
 For $a,b \in [n]$, we say the pair $\{a,a'\}$  is above $\{b,b'\}$ if $a < b$. Hence when $\{a,a',b,b'\}$ are drawn as boundary vertices of a plabic graph with our conventions, the elements of
 $\{a,a'\}$ indeed lie above the elements $\{b,b'\}$.  

\begin{lem}
$I$ is admissible if and only if no full pair of $I$ has both an empty pair above it and an empty pair below it; and no empty pair has both a full pair above it and a full pair below it.  
\end{lem}

\begin{proof}
Note that $\bar{I}$ is obtained from $I$ by replacing the full pairs of $I$ with empty pairs and vice versa. Hence $I \backslash \bar{I}$ is precisely the union of the full pairs of $I$, and $\bar{I} \backslash I$ is the union of the empty pairs of $I$. The lemma follows.
\end{proof}

We say $I \in {{[2n]}\choose{n}}$ is \emph{pair-free} if $I$ has no full pairs, and consequently no empty pairs.  Pair-free elements will be important in the proof of our main result.  Note that $\bar{I} = I$ if and only if $I$ is pair-free.  For a Grassmann necklace $\mathcal{I}=(I_1,\ldots,I_{2n})$ of type $C$, we have $\overline{I_1} = I_{2n + 1} = I_1$
so $I_1$ is pair-free.  
By a similar argument, so is $I_{n + 1}.$

\begin{lem}\label{pairless}Let $\mathcal{M}$ be a positroid of type $C$, with trip permutation $f$. Let $\mathcal{J}$ be a set of pair-free elements contained in a weakly separated collection in $\mathcal{M}$. Then the following hold:
\begin{enumerate}
\item For $I \in \mathcal{J}$, if $\{a, f^{-1}(a)\} \subseteq [n]$ and $f^{-1}(a) > a$, then $a \in I$.
\item For $I \in \mathcal{J}$, if $\{a,f^{-1}(a)\} \subseteq [n]$ and $f^{-1}(a) < a$, then $a' \in I$.
\item For $I, J \in \mathcal{J}$, either $(I \cap [n]) \subseteq (J \cap [n])$, or $(J \cap [n]) \subseteq (I \cap [n])$.
\end{enumerate}
\end{lem}

\begin{proof}
To prove $(1)$, observe that by symmetry, $(f^{-1}(a),f^{-1}(a'))$ is an alignment of $f$.  So by Corollary 11.4 in \cite{OPS15}, if $I$ contains $a'$, then $I$ would necessarily contain $a$ as well, contradicting the fact that $I$ is pair-free.  Hence $I$ must contain $a$, and not $a'$.  The proof of $(2)$ is analogous. For the last condition, let $I$ and $J$ be pair-free elements of ${{[2n]}\choose{n}}$.  Suppose we have $a,b \in [n]$ such that $a \in I \backslash J$ and $b \in J \backslash I$.  Then since $I$ and $J$ are pair-free, we have
$a' \in J \backslash I$ and $b' \in I \backslash J$.  Without loss of generality, say $a < b$. Then
\[a < b < b' < a'.\]
Since $\{a,b'\} \in I \backslash J$ and $\{b,a'\} \in J \backslash I$, $I$ and $J$ are not weakly separated.  The claim follows by contrapositive.
\end{proof}

For $f$ as above, let 
\begin{equation}
	S = \{a \in [n] \mid f^{-1}(a) > n\}.
\end{equation}

Let $r = |S| + 1$.  By Lemma \ref{pairless}, a symmetric weakly separated collection $\mathcal{C}$ in $\mathcal{M}$ contains at most $r$ pair-free elements.  If $\mathcal{C}$ contains $r$ pair-free elements, then they must be obtained by starting with $I_1$, and successively replacing $a$ with $a'$ with each $a \in S$ until we have $I_{n+1}$.  Note that the elements of $S$ may or may not be replaced in order.

\begin{defn}
Let $\mathcal{M}$, $\mathcal{I}=\{I_1,\ldots,I_{2n}\}$, $f$, and $r$ be as above.  We call a collection $\mathcal{J} \subseteq \mathcal{M}$ a \emph{spine} if $|J| = r$, and $J \cup \{I_1,\ldots,I_{2n}\}$
is a weakly separated collection.
\end{defn}

The collection of face labels of any symmetric plabic graph $G$ contains a spine, consisting of the faces of $G$ which intersect the midline. This follows, since the pairs of trips that cross at edges of $G$ which intersect the midline are precisely the pairs $(a,a')$, where $a \in [n]$ and $f^{-1}(a) > n$. See \cite{Kar18} for details.

\section{Symmetric plabic tilings}
\label{sym_tilings}

We fix the following conventions for plabic tilings.  The points $v_1,\ldots,v_{2n}$ are evenly spaced in counterclockwise order around a circle centered at the origin, which intersects the horizontal axis midway between $v_{2n}$ and $v_1$ on the left, and midway between $v_n$ and $v_{n+1}$ on the right.  See Figure \ref{numbers}.  

\begin{lem}
\label{reflect_vert}
With our choice of conventions, for an admissible set $I \in {{[2n]}\choose{n}}$, the reflection of $v_I$ across the vertical axis is $v_{\bar{I}}$.  The point $v_I$ lies on the vertical axis if and only if $I$ is pair-free; to the left of the vertical axis if $v_I$ has full pairs above empty pairs; and to the right otherwise.
\end{lem}

\begin{proof}
Let $I$ be admissible.  Then $\bar{I}$ is obtained from $I$ by exchanging full and empty pairs.
Note that both full and empty pairs contribute zero to the vertical component of the vector $v_I$, so $v_I$ and $v_{\bar{I}}$ differ only in their horizontal component.  It is enough to show that the horizontal component of $v_I + v_{\bar{I}}$ is zero.  Let $x_a$ be the horizontal component of $v_a$ for each $a \in [n]$.  Then each $x_a$ appears twice in the sum $v_{I} + v_{\bar{I}}$ since either $I$ and $\bar{I}$ each have a half pair at $\{a,a'\}$; or exactly one of $I$, $\bar{I}$ has a full pair.  
But $\sum_{a \in [n]}x_a = 0$, which proves the claim.

It follows that for $I$ admissible, $v_I$ lies on the vertical axis if and only if $v_I = v_{\bar{I}}$.  But since $I$ and $\bar{I}$ are weakly separated, $v_{I} = v_{\bar{I}}$ if and only if $I = \bar{I}$ \cite{OPS15}. Since $I = \bar{I}$ if and only if $I$ is pair-free, it follows that $v_I$ lies on the vertical axis if and only if $I$ is pair-free.

Since $\{a,a'\}$ is \emph{above} $\{b,b'\}$ if and only if $v_a$ and $v_{a'}$ appear to the \emph{left} of $v_b$ and $v_{b'}$, points $v_I$ appear on the \emph{left} side of the graph if $I$ has full pairs over empty pairs, and on the \emph{right} if $I$ has empty pairs over full pairs.

\end{proof}

\begin{lem}
\label{reflect_clique}
Let $\mathcal{M}$ be a positroid of type $C$, and let $\mathcal{C}$ be a symmetric weakly separated collection in $\mathcal{M}$, which is maximal by size (and hence gives the face labels of a plabic graph).
Let $L, K \subseteq [2n]$, where $|L| = n-1$ and $|K| = n+1$.
Then
\begin{enumerate}
\item \label{white_tri}
$\{\bar{I} \mid I \text{ is in the white clique of $L$}\}$
is the black clique of 
$\overline{L} \coloneqq [2n] \backslash \{\ell' \mid \ell \in L\}.$
\item \label{black_tri}
$\{\bar{J} \mid J \text{ is in the black clique of $K$}\}$
is the white clique of 
$\overline{K} \coloneqq [2n] \backslash \{k' \mid k \in K\}.$
\end{enumerate}
\end{lem}

\begin{proof}
We have $I = L \cup \{a\}$ for some $a \in [2n]$
if and only if $\bar{I} = \overline{L} \backslash \{a'\}.$
Hence $I \in {{[2n]}\choose{n}}$ is in the white clique of $L$ if and only if $\bar{I}$ is in the black clique of $\overline{L}$.
This proves \eqref{white_tri}, and the argument for \eqref{black_tri} is analogous.
\end{proof}

\begin{prop}
\label{goodtiling}
Let $\mathcal{M}$ be a positroid of type $C$, and let $\mathcal{C}$ be a symmetric weakly separated collection in $\mathcal{M}$ which is maximal by size.  Then $\mathcal{C}$ is the set of face labels of a symmetric plabic graph. 
\end{prop}

\begin{proof}
Since $\mathcal{M}$ is of type C, there exists a symmetric plabic graph with positroid $\mathcal{M}$, whose face labels give a symmetric weakly separated collection. So in particular, a symmetric weakly separated collection in $\mathcal{M}$ that is maximal by \emph{size} must be maximal by size among \emph{all} weakly separated collections in $\mathcal{M}$, and hence must be the collection of face labels of a plabic graph.  We may therefore construct a plabic tiling $\Sigma(\mathcal{C})$, and the planar dual of $\Sigma(\mathcal{C})$ is a plabic graph $G$.  Our task is to show that $G$ is symmetric.

By Lemma \ref{reflect_vert} and Lemma \ref{reflect_clique}, the plabic tiling $\Sigma(\mathcal{C})$ corresponding to $\mathcal{C}$ is symmetric about the vertical axis, up to reversing the colors of tiles. Hence $G$, embedded as the planar dual to $\Sigma(\mathcal{C})$, is symmetric about the vertical axis, up to reversing the colors of vertices. 

Let $\mathcal{I} = (I_1,\ldots,I_{2n})$ be the Grassmann necklace of $\mathcal{M}$.  
Then $I_1$ and $I_{n+1}$ are both pair-free, and hence
$v_{I_1}$ and $v_{I_{n+1}}$ lie on the vertical axis. But $I_{1}$ is the label of the boundary face between legs $2n$ and $1$ of $G$, while $I_1$ is the label of the boundary vertex between legs $2n$ and $1$.
Hence the vertical axis cuts the boundary disk of $G$ between vertices $2n$ and $1$, and again between boundary vertices $n$ and $n+1$.  We note that $I_{n+1}$ is obtained from $I_1$ by replacing $j$ with $j'$ for each 
$j \in [n]$ such that $f^{-1}[j] > n,$
and so in particular $v_{I_1}$ is above $v_{I+1}$ in the tiling. It follows that $G$ is a symmetric plabic graph, which is embedded according to our conventions, and the proof is complete.

\end{proof}

\section{Proof of the Main Result}

\begin{lem}\label{bar}
Let $I, J \in {{[2n]}\choose{n}}$.  Then $I$ is weakly separated from $J$ if and only if $\bar{I}$ is weakly separated from $\bar{J}$.
\end{lem}

\begin{proof}
Suppose $I$ is not weakly separated from $J$.  Then there exist $a,b \in I\backslash J$ and $x,y \in J \backslash I$, such that $a,x,b$ and $y$ are cyclically ordered.
Now, since $a,b \in I \backslash J$, we know that $a',b' \in \bar{J} \backslash \bar{I}$. Similarly, $x',y' \in \bar{I} \backslash \bar{J}$. 
 Moreover, $y',b',x'$ and $a'$ are cyclically ordered, since if the elements of $[2n]$ are arranged around a circle, the sequence $a', x', b', y'$ is obtained from $a, x, b, y$ by a reflection.
So $\bar{I}$ is not weakly separated from $\bar{J}$.  The lemma follows by contrapositive, together with the fact that $\bar{\bar{I}} = I$.
\end{proof}

\begin{lem}\label{spine}Let $\mathcal{C}$ be symmetric weakly separated collection in $\mathcal{M}$, which is maximal by inclusion.  Then $\mathcal{C}$ contains a spine.
\end{lem}

\begin{proof}
Let $\mathcal{I} = (I_1,\ldots,I_{2n})$ be the Grassmann necklace of $\mathcal{M}$, and suppose $\mathcal{C}$ does not contain a spine.  Then $\mathcal{C}$ is not maximal by size, since otherwise $\mathcal{C}$ would be the set of face labels of a symmetric plabic graph.  Extend $\mathcal{C}$ to a maximal weakly separated collection $\mathcal{C}'$ in $\mathcal{M}$.  By maximality, $\mathcal{C}'$ cannot contain any pair-free  elements which are not found in $\mathcal{C}$, and in particular $\mathcal{C}'$ does not contain a spine.  Hence, $\mathcal{C}$ must contain some pair-free element $I \neq I_{n+1}$, such that no pair-free element of the form $I \backslash \{a\} \cup \{a'\}$ with $a \in [n]$ is contained in $\mathcal{C}$.

Next, triangulate each two-dimensional tile of $\Sigma(\mathcal{C}')$.  
We claim there is a triangle $T$ in the resulting \emph{triangulated} plabic tiling which contains a segment of the midline immediately below $I$, and has $v_I$ as its top vertex. We first show that the entire segment from $v_{I_{1}}$ to $v_{I_{n+1}}$ is contained in the union of the tiles of $\Sigma(\mathcal{C}')$, and so in particular the segment of the midline immediately below $v_{I}$ is contained in the union.  

Let $\mathcal{C}^*$ be the set of face labels of a \emph{symmetric} plabic graph for $\mathcal{M}$.  Then $\mathcal{C}^*$ contains a spine, and so in particular the union of the tiles of $\mathcal{C}^*$ contains the desired line segment.  But then the same must be true true for $\mathcal{C}'$, because every plabic tiling for $\mathcal{M}$ is related by a series of local transformations called \emph{mutations}, which leave the \emph{union} of the tiles unchanged \cite{OPS15}.  Hence the portion of the midline immediately below $v_I$ must be contained in either an edge of the triangulated complex, or the interior of a triangle tile.

Note that $v_I$ is pair-free, and any neighbor of $I$ differs from $I$ in exactly two places.  Hence every neighbor of $I$ has at most one full pair and one empty pair, and is therefore admissible. If the segment of the midline immediately below $v_I$ were contained in an edge of the triangulated complex, then $I$ would have an admissible neighbor below it, which would necessarily be a pair-free element of the form $(I \backslash \{a\}) \cup \{a'\}$ for some $a \in [n]$.  This contradicts our hypothesis.  

Hence $T$ is the top-most vertex of a triangle, with the remaining two vertices lying on either side of the midline.  Let $J_1$ and $J_2$ be the other two vertices of $T$, where $J_1$ lies on the left side of the midline and $J_2$ on the right.  Without loss of generality, we may suppose the triangle $T$ is white. If not, we simply replace $\mathcal{C}'$ with $\{\bar{J} \mid J \in \mathcal{C}'\}$.

Note that $J_1$ and $J_2$ each have only one full pair and one empty pair.  Let
\[I' = I \cap J_1 \cap J_2\]
and suppose $I = I' \cup \{a\}$.  
It follows that $J_1$ and $J_2$ each have an empty pair at $\{a,a'\}$, $J_1$ has a full pair above $\{a,a'\}$, and $J_2$ has a full pair below. 

We claim that $\bar{J_1}$  is not weakly separated from $J_2$. Indeed, suppose $J_1$ has a full pair $\{b,b'\}$ while $J_2$ has a full pair $\{c, c'\}$.  Then some $b^* \in \{b, b'\}$ and some $c^* \in \{c,c'\}$ are contained in $J_2 \backslash \bar{J_1}$, while $\{a,a'\} \subset \bar{J_1} \backslash J_2$.  
Since $b^*$ and $c^*$ are on opposite sides of the chord connecting $a$ and $a'$ when the elements of $[2n]$ are arranged around a circle.  Hence there is no chord separating $\{b^*, c^*\}$ from $\{a,a'\}$, so
$\bar{J_1}$ and $J_2$ are not weakly separated, as desired.   
Similarly, $\bar{J_2}$ is not weakly separated from $J_1$.

Now, both $J_1$ and $J_2$ are weakly separated from all elements of $\mathcal{C}$, and both are admissible.  It follows that neither $J_1$ nor $J_2$ is contained in $\mathcal{C}$; indeed, if $J_1 \in \mathcal{C}$, then so is $\bar{J_1}$; contradicting the fact that $J_2$ is weakly separated from every element of $\mathcal{C}$. Similarly, we cannot have $J_2 \in \mathcal{C}$. But then we can add $\{J_1, \bar{J_1}\}$ to $\mathcal{C}$, and obtain a symmetric weakly separated collection in $\mathcal{M}$.  This contradicts the maximality of $\mathcal{C}$, and the proof is complete.
\end{proof}

\begin{lem}\label{ifspine} 
Let $\mathcal{C}$ be a symmetric weakly separated collection in $\mathcal{M}$, where $\mathcal{M}$ is a positroid of type $C$.  If $\mathcal{C}$ is maximal by inclusion, then $\mathcal{C}$ is maximal by size.
\end{lem}

\begin{proof}
By Lemma \ref{spine}, we can assume $\mathcal{C}$ contains a collection $\mathcal{J}$ of $r$ pair-free elements. Suppose that $\mathcal{C}$ is not maximal by size.  Extend $\mathcal{C}$ to a maximal weakly separated collection $\mathcal{C}'$, which of necessity is not symmetric.  Let $I \in \mathcal{C}' \backslash \mathcal{C}.$  Then $I$ is weakly separated from each element of $\mathcal{C}$. We claim that $\bar{I}$ is also weakly separated from each element of $\mathcal{C}$, and that $I$ is admissible. But then $\mathcal{C} \cup \{I,\bar{I}\}$ is a symmetric weakly separated collection in $\mathcal{M}$ which properly contains $\mathcal{C}$, a contradiction.

By symmetry, for each such $J \in \mathcal{C}$, we have $\bar{J} \in \mathcal{C}$.  So $I$ is weakly separated from $\bar{J}$, and by Lemma \ref{bar}, $\bar{I}$ is weakly separated from $J$. Our task is to show that $I$ and $\bar{I}$ are weakly separated.  

Suppose $I$ is not weakly separated from $\bar{I}$. Then there is no chord separating the full pairs of $I$ from the empty pairs.  So there must exist some
\[1 \leq a < b < c \leq n\]
where either $\{a,a'\}$ and $\{c,c'\}$ are full pairs of $I$, and $\{b,b'\}$ is an empty pair; or $\{a,a'\}$ and $\{c,c'\}$ are empty pairs of $I$ and $\{b,b'\}$ is a full pair.  Interchanging $I$ and $\bar{I}$ if necessary, we may assume the first case holds.  We claim that $f^{-1}(b) \leq n$.  Suppose not.  Then there exist pair-free elements $J_1,J_2 \in \mathcal{C}$, such that $J_2 =( J_1 \backslash \{b\}) \cup \{b'\}$.  By construction, $I$ is weakly separated from both $J_1$ and $J_2$.  Note that for $J \in \mathcal{J}$, each full pair of $I$ contributes an element to $I \backslash J$, and each empty pair contributes an element of $J\backslash I$.  In particular, there exist $a^* \in \{a,a'\}$, $c^* \in \{c,c'\}$ such that $a^*, c^* \in I \backslash J_1$.

Since $J_1$, $J_2,$ and $I$ are all $n$-element subsets of $[2n]$, there must be at least one $x \in [2n]$ such that $x,x' \not\in \{a,b,c\}$, and $x \in J_1 \backslash I$.  Consider the chord from $a^*$ to $c^*$.  Since $x,b \in J_1 \backslash I$, and $a^*, c^* \in I \backslash J_1$,  it follows that $x$ and $b$ are on the same side of this chord.  But the chord from $a^*$ to $c^*$ crosses the chord from $b$ to $b'$, so $b'$ and $x$ are on opposite sides of the chord from $a^*$ to $c^*$.  This, in turn, implies that $I$ is not weakly separated from $J_2$, a contradiction.

Hence, we must have $f^{-1}(b) \leq n$.  
Suppose $f^{-1}(b) < b$. By symmetry, $f^{-1}(b') > b'$.  
It follows that $f^{-1}(c)$ and $f^{-1}(b')$ form an alignment unless $ f^{-1}(c) \in [c, f^{-1}(b')]^{cyc}$.  
But if the latter occurs, then by symmetry $f^{-1}(c') \in [f^{-1}(b), c']^{cyc}$, so $f^{-1}(c')$ and $f^{-1}(b')$ form an alignment.  In either case, by \cite[Corollary 11.3]{OPS15}, any element of $\mathcal{C}$ which contains $\{c,c'\}$ must also contain $b'$.  This contradicts our assumption on $I$, and the proof is complete in this case.

If $f^{-1}(b) > b$, the argument is similar.  Either $f^{-1}(a)$ and $f^{-1}(b)$ form an alignment, if $f^{-1}(a) \not\in [a, f^{-1}(b)]^{cyc}$, or else $f^{-1}(a')$ and $f^{-1}(b)$ form an alignment. In either case we get a contradiction.  This completes the proof.

\end{proof}

\begin{thm}
\label{main}
Let $\mathcal{C}$ be a symmetric weakly separated collection with $\mathcal{I} \subseteq \mathcal{C} \subseteq \mathcal{M}$.  Then $\mathcal{C}$ is maximal by inclusion if and only if $\mathcal{C}$ is the set of face labels of a symmetric plabic graph.  
\end{thm}

\begin{proof}
The face labels of a symmetric plabic graph form a symmetric weakly separated collection in $\mathcal{M}$ which is maximal by size, and so must be maximal by inclusion. Conversely, let $\mathcal{C}$ be a symmetric weakly separated collection in $\mathcal{M}$, which is maximal by inclusion.  By Lemma \ref{ifspine}, it follows that $\mathcal{C}$ is in fact maximal by size, and is hence the set of face labels of a symmetric plabic graph, by Proposition \ref{goodtiling}.
\end{proof}  

\begin{cor}Let $\mathcal{M}$ be a positroid of type $C$, with decorated permutation $f$.  Then $I \in \mathcal{M}$ is a face label of some symmetric plabic graph for $\mathcal{M}$ if and only if the following conditions hold:
\begin{enumerate}
\item \label{ops_cor} For all $i,j \in [2n]$, if $(i,j)$ forms an alignment of $f$ and $f(i) \in I$, then $f(j) \in I$.
\item $I$ is admissible.
\end{enumerate}

\end{cor}

\begin{proof}
Let $\mathcal{I} = (I_1, I_2,\ldots,I_{2n})$ be the Grassmann necklace of $\mathcal{M}$.  By \cite[Corollary 11.4]{OPS15}, $I$ is weakly separated from each element of $\mathcal{I}$ if and only if \eqref{ops_cor} holds.  
Since $\mathcal{I}$ is type $C$, $\{I_1,\ldots,I_{2n}\}$ is a symmetric weakly separated collection. 
Hence $\{I_1,\ldots,I_{2n}\} \cup \{I, \bar{I}\}$ is a weakly separated collection in $\mathcal{M}$ if and and only if $I$ is admissible and weakly separated from $\{I_1,\ldots,I_{2n}\}$.
The result follows by Theorem \ref{main}.

\end{proof}

Call an admissible set $I$ \emph{left-handed} if it has full pairs over empty pairs; and hence falls to the left of the midline in a symmetric plabic tiling.  Then we can restate Theorem \ref{main} in terms of total positivity tests. 

\begin{cor}
\label{pos_tests}
Let $\mathcal{C}$ be a collection of elements of ${{[2n]}\choose{n}}$, which satisfies the following conditions:
\begin{enumerate}
\item \label{first} All elements of $\mathcal{C}$ are admissible.
\item Every element of $\mathcal{C}$ is either pair-free or left-handed.
\item \label{last} For each $I, J \in \mathcal{C}$, the set $I$ is weakly separated from both $J$ and $\bar{J}$.
\end{enumerate}
Then $\mathcal{C}$ is maximal by \emph{inclusion} among collections that meet condition \eqref{first}-\eqref{last} if and only if 
\[|\mathcal{C}| = \frac{n^2 + n + 2}{2}.\]  
If this holds, the Pl\"{u}cker coordinates indexed by $\mathcal{C}$ give a total positivity test for $\Lm(2n)$, which is \emph{minimal} by inclusion; removing any element of $\mathcal{C}$ leaves a collection of Pl\"{u}cker coordinates which is \emph{not} a total positivity test.
\end{cor}

\begin{proof}

We note that $\mathcal{C}' \coloneqq \mathcal{C} \cup \{\bar{I} \mid I \in \mathcal{C}\}$
is a symmetric weakly separated collection.  Moreover $\mathcal{C}$ is maximal by inclusion among sets that satisfy \eqref{first}-\eqref{last} if and only if $\mathcal{C}'$ is maximal as a symmetric weakly separated collection.

A maximal weakly separated collection in ${{[2n]}\choose{n}}$ has $n^2 + 1$ elements total.  If the weakly separated collection is symmetric, then $n+1$ elements form a spine, while $\frac{n(n-1)}{2}$ are left-handed. Hence by Lemma \ref{ifspine}, a symmetric weakly separated collection has at most 
$\frac{n^2 + n + 2}{2}$ faces that are either left-handed or pair-free. A symmetric weakly separated collection with the maximal number of elements that are either left-handed or pair-free must be maximal, since at most $n+1$ of those elements are pair-free, and the total size of the collection is at most $n^2 + 1$.  Hence $\mathcal{C}$ is maximal by inclusion if and only if $|C| = \frac{n^2 + n + 2}{2}$.
The statement about positivity is immediate from the discussion in Section \ref{sym_face}.
\end{proof}

\begin{rmk}It is natural to ask whether the conditions of Corollary \ref{pos_tests} are unnecessarily strict; and in particular whether \emph{any} weakly separated collection of admissible, left-handed elements may be extended to give the face labels of a symmetric plabic graph.
To see that this is not the case, let $I = \{1,3,6\}$, and $J = \{1,4,6\}$. Then $I$ and $J$ are admissible, left-handed elements of ${{[6]}\choose{3}}$, which are weakly separated from each other.
But $\bar{I} = \{2,3,5\}$, so $\{2, 5\} \subseteq \bar{I}\backslash J$, while $\{1,4\} \subseteq J \backslash \bar{I}$.  Since the sequence $1,2, 4, 5$ is cyclically ordered, $\bar{I}$ is not weakly separated from $J$, and there is no symmetric, maximal weakly separated collection which contains both $I$ and $J$.
\end{rmk}


\bibliographystyle{plain}
\bibliography{weaksepC.bib}
\end{document}